\documentclass{birkmult}
 \newtheorem{theorem}{Theorem}[section]
 \newtheorem{corollary}[theorem]{Corollary}
 \newtheorem{lemma}[theorem]{Lemma}
 \newtheorem{proposition}[theorem]{Proposition}
 \newtheorem{conjecture}[theorem]{Conjecture}
 \theoremstyle{definition}
 
 \theoremstyle{remark}
 
 \newtheorem{example}{Example}
 \numberwithin{equation}{section}

\DeclareMathOperator{\im}{Im}

\DeclareMathOperator{\supp}{\mathrm{supp}}

\DeclareMathOperator*{\curl}{\mathrm{curl\,}}
\DeclareMathOperator*{\di}{\mathrm{div \,}}

\def\R#1{\mathbb{R}^{#1}}

\def\I{\mathrm{i}}

\def\R{{\mathbb R}}
\def\C{{\mathbb C}}

\def\x{{\mathbf x}}
\def\y{{\mathbf y}}

\begin{document}

\title{On the curvature of some free boundaries in higher dimensions}

\author{Bj\"orn Gustafsson}

\address{Department of Mathematics\\
KTH\\ 100 44 Stockholm
\\Sweden}

\email{gbjorn@kth.se}

\author{Makoto Sakai}

\address{Department of Mathematics\\
Tokyo Metropolitan University\\
Minami-Ohsawa\\
Hachioji-shi,\\
Tokyo 192-0397, Japan}

\email{sakai@tmu.ac.jp}

\date{December 16, 2011}

\thanks{Paper supported by Swedish Research Council, the Mittag-Leffler Institute,
the G\"oran Gustafsson Foundation, and the European Science Foundation Research Networking Programme HCAA }


\begin{abstract}
It is known that any subharmonic quadrature domain
in two dimensions satisfies a natural inner ball
condition, in other words there is a specific upper bound
on the curvature of the boundary. This result directly
applies to free boundaries appearing in obstacle
type problems and in Hele-Shaw flow. In the present paper
we make
partial progress on the corresponding question in higher
dimensions. Specifically, we prove the equivalence between
several different ways to formulate the inner ball condition,
and we compute the Brouwer degree for a geometrically
important mapping related to the Schwarz potential of the boundary. The latter gives in particular a new proof in the two dimensional case.

\end{abstract}


\maketitle


\section{Introduction}

In the present paper we study the curvature of a
free boundary which comes up in some obstacle type problems, more specifically
Laplacian growth (Hele-Shaw flow moving boundary problem),
quadrature domains and partial balayage. The final aim of the investigations is to show
that the free boundary in question satisfies a certain inner ball condition (a specific upper bound
on the curvature). This goal has previously been achieved in the case of two dimensions (see \cite{Gustafsson-Sakai 03} and
\cite{Gustafsson-Sakai 04}; compare also \cite{Schaeffer 77}).
Here we shall give some partial results (but no complete solution) in higher dimensions, and in passing also obtain a new proof for the two dimensional case.

The geometric property we aim at proving can most easily be stated in terms of quadrature domains for subharmonic
functions \cite{Sakai 82}. Let $\mu$ a positive Borel measure with compact support in $\R^n$.
A bounded open set $\Omega\subset \R^n$ is called a {\it quadrature open set} for
subharmonic functions with respect to $\mu$ if $\mu(\R^n\setminus\Omega)=0$ and
$$
\int_\Omega h \,d\mu\leq \int_\Omega h\,dm
$$
for all integrable (with respect to Lebesgue measure $dm$) subharmonic functions $h$ in $\Omega$.
A quadrature open set which is connected is a {\it quadrature domain}. There is a natural
process of balayage of measures to a prescribed density (partial balayage) by which $\Omega$
can be constructed from $\mu$ when it exists, see e.g. \cite{Sakai 82}, \cite{Gustafsson-Sakai 94},
\cite{Gustafsson 90}.
This balayage process is also equivalent to solving a certain obstacle problem \cite{Sakai 83}.

In terms of the difference $u=U^\mu -U^\Omega$ between the Newtonian potentials of $\mu$ and $\Omega$,
the latter considered as a body of density one, the quadrature property spells out to
$$
\begin{cases}
u\geq 0 \quad \text{in } \R^n,\\
u=0 \quad \text{outside } \Omega,\\
\Delta u = \chi_\Omega -\mu \quad \text{in } \R^n.
\end{cases}
$$
The function $u$ appearing here is sometimes called the modified Schwarz potential of $\partial\Omega$,
see \cite{Shapiro 92} for example.
What has been proved in two dimensions, and what we like to extend to higher dimension,
is that $\Omega$ in the above situation can be written as the union of open balls centered
in the closed convex hull $K$ of ${\rm supp\,}\mu$:
\begin{equation}\label{innerball}
\Omega=\cup_{x\in K} B(x,r(x)).
\end{equation}
Here $r(x)\geq 0$ denotes the radius of the ball at $x$, allowing the possibility $r(x)=0$,
i.e., that the ball is empty. We refer to (\ref{innerball})
as $\Omega$ satisfying the inner ball property with respect to $K$.
For further discussion and motivations in the present context, see \cite{Gustafsson-Sakai 03}, \cite{Gustafsson-Sakai 04}.

The inner ball property concerns the geometry of $\partial\Omega$ outside any closed half-space $H$
containing ${\rm supp\,}\mu$. By a certain ``localization'' procedure the part $\Omega\setminus H$ of
$\Omega$ which is outside $H$ can be shown to be identical with a quadrature open set for some positive measure with support
on $\partial H$, see \cite{Gustafsson-Sakai 94}, \cite{Sakai 10}. For this reason it is enough to
prove (\ref{innerball}) in the case that $\mu$ has support in a hyperplane, and by a further localization
one may even assume $\mu$ to have a continuous density on it. For convenience we shall
take the hyperplane in question to be $\{x\in\R^n: x_n=0\}$.
It is known (see \cite{Sakai 82}, \cite{Gustafsson-Sakai 94}) that given any positive measure with compact
support in this hyperplane, which we identify with $\R^{n-1}$, there is a uniquely determined quadrature
open set, which moreover is symmetric about the hyperplane and is convex in the $x_n$-direction
(i.e., the intersection with any straight line perpendicular to the hyperplane is connected). Thus the quadrature
open set can be described in terms of a graph of a function $g$ defined in an open subset $D$ of the hyperplane. It is known
from the regularity theory of free boundaries (see \cite{Friedman 82}, \cite{Caffarelli 98}, and in
the present context \cite{Gustafsson-Sakai 94}) that this function $g$ is real analytic.
It will be enough to consider the case that $D$ is connected, because in the disconnected case
the discussions will apply to each component separately.

In the paper we shall therefore study the geometry of quadrature domains for a measure with support
in the hyperplane $\R^{n-1}$, the corresponding modified Schwarz potential $u$, and various
differential geometric objects derived from it. The paper consists of two main parts.
The first part starts with some differential geometric preliminaries (Section~\ref{sec:diffgeom})
and ends up with a proof of equivalence of
several different formulations of the inner ball condition (Section~\ref{sec:innerball}). This part is analogous to a corresponding
part in \cite{Gustafsson-Sakai 03}, but new difficulties appear in the higher dimensional case. In the second part
of the paper, Section~\ref{sec:boundaries},
we develop tools for studying the geometry by means of
vector fields and differential forms defined in terms of $u$.

To briefly explain, in terms the two dimensional situation, what we do in the second
part of the paper, let $\Omega^+$ denote the part of the quadrature domain which lies in
the upper half space (half plane) and let $S(z)$ be the Schwarz function (see \cite{Davis 74},
\cite{Shapiro 92}) of $(\partial\Omega)^+$.
This is in our case given by $S(z)=\bar z-4\frac{\partial u}{\partial z}$, it is analytic in $\Omega^+$,
and equals $\bar z$ on $(\partial\Omega)^+$. What we study in the second part of the paper is the higher
dimensional counterpart of the mapping $\sigma: \Omega^+\to \C$ defined by
$$
\sigma (z)= \frac{1}{2}zS(z)=\frac{1}{2}r^2-r\frac{\partial u}{\partial r} +\I \frac{\partial u}{\partial \varphi},
$$
the last member referring to polar coordinates. We also study in higher dimensions that curve $\gamma$
which in two dimensions is defined by $\frac{\partial u}{\partial \varphi}=0$. In the two dimensional case a proof of the
inner ball property can be based on
the topological property of $\gamma$ that it separates the two domains
$\frac{\partial u}{\partial \varphi}<0$ and $\frac{\partial u}{\partial \varphi}>0$ from each other, see
\cite{Gustafsson-Sakai 04}. There seems to be no direct counterpart of this kind of proof in higher dimensions.

A slightly different proof in two dimensions uses the argument principle for $\sigma$,
see Corollary~\ref{cor:main} in the present paper. We have not been able to generalize this proof either to higher
dimensions, even though we think that such a proof may not be completely out of reach. At least we have
computed the relevant mapping degree of $\sigma$ in higher dimensions, and this result, Theorem~\ref{thm:main},
may be considered to be the main result in the second part of the paper. Another avenue which we have pursued
to some extent is the investigation of the general behaviour of the curve $\gamma$. This curve starts out at the origin
and reaches $(\partial\Omega)^+$ only at points where the largest inner ball centered at the origin touches
$(\partial\Omega)^+$, see Proposition~\ref{prop:gamma}. If we could establish for example that $\gamma$
were a smooth curve (no branchings) a proof of the inner ball property would not be far away.

Thus we hope that the partial results we obtain in this paper will turn out to be useful in forthcoming
attempts to prove the inner ball property for quadrature domains, explicitly formulated in
Conjecture~\ref{conj:main}.

\section{Differential geometric preliminaries}\label{sec:diffgeom}

\subsection{Notations}

In this section we decompose the coordinates in $\R^n$ typically as $(u,v)$ where
$u=(u_1,...,u_{n-1})\in \R^{n-1}$, $v\in \R$.
We denote by $(\R^n)^+=\{(u,v)\in\R^n : v>0\}$ the upper half space.
Let $g$ be a positive, twice continuously differentiable function
defined in a bounded domain $D\subset \R^{n-1}$ and assume that
$g(u)\to 0$ as $u\to \partial D$
(the boundary as a subset of $\R^{n-1}$).
It is convenient to set $g(u)=0$ for $u\notin D$, and then
$g$ is defined and continuous on all $\R^{n-1}$.
Set
\begin{equation}\label{Omega}
\Omega =\{(u,v)\in \R^{n}: u\in D , \, |v|<g(u) \},
\end{equation}
and for any subset $A\subset \R^n$, $A^+= A\cap (\R^n)^+$.

The normal line of $(\partial\Omega)^+$ at a point $(u,g(u))$
is given in parametrized form as
$$
t\mapsto (u,g(u)) +t(\nabla g(u), -1))
$$
and it intersects the hyperplane $\R^{n-1}$ for $t=g(u)$,
i.e., at the {\em foot point}
$$
p(u) = u + g(u) \nabla g (u).
$$
Let
$$
N_u =\{(u,g(u)) +t(\nabla g(u), -1)): 0<t< g(u)\}
$$
be the part of the normal line which is between the base point on
$(\partial\Omega)^+$ and the foot point. The length of $N_u$ is
$$
|N_u | = g(u)\sqrt{1+|\nabla g(u) |^2}.
$$

For any point $x\in\Omega^+$ there is at least one closest point $(u,g(u))$ on
$(\partial\Omega)^+$. Then $x\in N_u$ and, in particular,
$x\in B(p(u), |N_u|)$.
Note that $N_u$ is one of the radii in $B(p(u), |N_u|)$.
It follows that we always have inclusions

\begin{equation*}
\Omega^+\subset\bigcup_{u\in D} N_u
\subset \bigcup_{u\in D} B(p(u), |N_u|).
\end{equation*}
Also,
\begin{equation}
\Omega\subset \bigcup_{u\in D} B(p(u), |N_u|).
\label{OmegaB}\end{equation}


\subsection{The first and second fundamental forms}\label{sec:fundamentalforms}

We shall discuss $(\partial\Omega)^+$ from a differential geometric point
of view. We consider it as  a Riemannian
manifold of dimension $n-1$ and with coordinates
$u_1, ...,u_{n-1}$. The Riemannian metric is that inherited
from $\R^n$.
We shall write certain quantities considered as vectors in
$\R^n$ in bold.
These involve the ``moving point'' on $(\partial\Omega)^+$
(or lift map $D\to (\partial\Omega)^+$)
$$
{\bf x} ={\bf x}(u) = (u,g(u)),
$$
its differential (a vector-valued differential form)
$$
d{\bf x} = (du_1, ...,du_{n-1}, \sum_{i=1}^{n-1}
\frac{\partial g(u)}{\partial u_i} du_i),
$$
the unit normal vector
$$
{\bf n} ={\bf n} (u)=\frac{(\nabla g(u) , -1)}{\sqrt{|\nabla g(u)|^2 +1}}
$$
and its differential $d{\bf n}$ (which becomes
somewhat complicated when written out in components).
For a point $x\in\Omega^+$ we sometimes write it in bold if we think of it
as the vector from the origin to $x$.

A tangent vector $\boldsymbol{\xi}$ on $(\partial\Omega)^+$
may be thought of in an  abstract way as a derivation:
$$
\boldsymbol{\xi} =\sum_{i=1}^{n-1}{\xi}_i\frac{\partial}{\partial u_i}.
$$
Letting this $\boldsymbol{\xi}$ act on the moving point ${\bf x}(u)$
or, equivalently, letting the differential $d{\bf x}$ act on
$\boldsymbol{\xi}$ gives the same tangent vector regarded as a vector
embedded in $\R^n$:
$$
\langle d{\bf x}, \boldsymbol{\xi} \rangle
=(\xi_1,...,\xi_{n-1},
\sum_{i=1}^{n-1}\xi_i \frac{\partial g}{\partial u_i}).
$$

In classical differential geometry
(see for example \cite{Frankel 97}) one associates to any
hypersurface in Euclidean space two fundamental forms.
The first fundamental form is the metric tensor, which
gives the inner product on each tangent space. It is in our case
\begin{align*}
ds^2 = d{\bf x}\cdot d{\bf x}
&=du_1^2+...+du_{n-1}^2 +(\sum_{i=1}^{n-1}
\frac{\partial g}{\partial u_i} du_i)^2 \\
&=\sum_{i,j=1}^{n-1} (\delta_{ij} + \frac{\partial g}{\partial u_i}
\frac{\partial g}{\partial u_j}) du_i \otimes du_j,
\end{align*}
where the dot denotes the scalar product in $\R^n$. The second
fundamental form is (up to a sign)
$$
d{\bf x}\cdot d{\bf n}
=\frac{1}{\sqrt{1+|\nabla g|^2}}\sum_{i,j=1}^{n-1}
\frac{\partial^2 g}{\partial u_i \partial u_j}du_i\otimes du_j.
$$
It contains information on
how ${\bf n}$ rotates in different directions,
i.e.  on how  $(\partial\Omega)^+$ is curved within $\R^n$.
The above expression for $d{\bf x}\cdot d{\bf n}$
can be derived as follows: Since ${\bf n}$
is orthogonal to $(\partial\Omega)^+$ we have
$\frac{\partial {\bf x}}{\partial u_i}\cdot {\bf n}=0$ for all $i$. Therefore
\begin{align*}
\frac{\partial {\bf x}}{\partial u_i}\cdot
\frac{\partial {\bf n}}{\partial u_j}
&=\frac{\partial }{\partial u_j}
(\frac{\partial {\bf x}}{\partial u_i}\cdot {\bf n})
-\frac{\partial^2 {\bf x}}{\partial u_i \partial u_j}\cdot {\bf n}
=0- (0,\frac{\partial^2 g}{\partial u_i \partial u_j})\cdot {\bf n} \\
&=\frac{1}{\sqrt{1+|\nabla g|^2}}
\frac{\partial^2 g}{\partial u_i \partial u_j},
\end{align*}
proving the formula.

Both fundamental forms are symmetric covariant 2-tensors,
i.e., symmetric bilinear forms on each tangent space.
When considered as bilinear forms
we shall call them $A$ and $B$ respectively
(namely $A=d{\bf x}\cdot d{\bf x}$,
$B=d{\bf x}\cdot d{\bf n}$).
The above formulas then
mean that for (abstract) tangent vectors $\boldsymbol{\xi}$ and $\boldsymbol{\eta}$ we have
\begin{align*}
A (\boldsymbol{\xi},\boldsymbol{\eta})
&=\sum_{i=1}^{n-1} \xi_i \eta_i
+\sum_{i,j=1}^{n-1}\frac{\partial g}{\partial u_i}
\frac{\partial g}{\partial u_j}\xi_i \eta_j, \\
B(\boldsymbol{\xi},\boldsymbol{\eta})
&=\frac{1}{\sqrt{1+|\nabla g|^2}}
\sum_{i,j=1}^{n-1}\frac{\partial^2g}{\partial u_i\partial u_j}\xi_i\eta_j.
\end{align*}
Note that $A$ is positive definite.

The principal curvatures of $(\partial\Omega)^+$ and the
corresponding principal directions are the eigenvalues and
eigendirections of the derivative of the Gauss normal map, taking
any point ${\bf x}\in(\partial\Omega)^+$ onto the normal vector at
the point: ${\bf n}({\bf x})\in S^{n-1}$; the derivative is the
induced linear map between the tangent spaces, intuitively $ d{\bf
x} \mapsto d{\bf n}$. In terms of abstract  tangent vectors this
becomes $C: \boldsymbol{\xi} \mapsto \boldsymbol{\eta}$, where
$\langle d{\bf x}, \boldsymbol{\eta}\rangle =\langle d{\bf n},
\boldsymbol{\xi}\rangle$, and one can also say that it is the map
obtained when expressing the second fundamental form in terms of the
first:
$$
A(\boldsymbol{\zeta} ,C\boldsymbol{\xi}
)=B(\boldsymbol{\zeta},\boldsymbol{\xi}).
$$

It is well-known (and easy to prove) that these eigenvalues
and eigendirections coincide with the
stationary points and corresponding stationary directions for the
quotient of the fundamental forms, namely for the map
$$
\boldsymbol{\xi}\mapsto\frac{B(\boldsymbol{\xi},\boldsymbol{\xi} )}{A(\boldsymbol{\xi}, \boldsymbol{\xi})}
$$
($\boldsymbol{\xi} \in \R^{n-1}$, $\boldsymbol{\xi}\ne 0$). For example, the smallest eigenvalue
(the smallest principal curvature) coincides with the minimum value of
$B/A$.


\subsection{The function $\Phi$}\label{Phi}

Next we introduce the function
$$
\Phi(u)=\frac{1}{2} (|u|^2 + g(u)^2)=\frac{1}{2}|{\bf x} (u) |^2, \quad (u\in D)
$$
namely half of the squared distance from points on $(\partial\Omega)^+$
to the origin.
More generally,
for any $c=(a,b)$ with $a\in\R^{n-1}$, $b\geq 0$ we set
$$
\Phi_c (u) =\frac{1}{2}(|u-a|^2 + (g(u)-b)^2)
=\frac{1}{2}|{\bf x} (u) -c|^2,
$$
considered to be defined for those $u\in D$ for which $g(u)>b$.
Then $\Phi_c$ is twice continuously differentiable with
$$
\frac{\partial \Phi_c }{\partial u_i}
= u_i-a_i + (g(u)-b)\frac{\partial g}{\partial u_i},
$$
\begin{align}
\frac{\partial^2 \Phi_c}{\partial u_i\partial u_j}
&= \delta_{ij}
+\frac{\partial g}{\partial u_i} \frac{\partial g}{\partial u_j}
+(g(u)-b)\frac{\partial^2 g}{\partial u_i\partial u_j } \label{nabla2Phi}\\
&=\frac{\partial^2 \Phi}{\partial u_i\partial u_j}
 - b \frac{\partial^2 g}{\partial u_i\partial u_j}.\label{nabla3Phi}
\end{align}

For $c=0$ we can write
$$
\frac{\partial^2 \Phi}{\partial u_i \partial u_j}
= \delta_{ij}
+\frac{\partial g}{\partial u_i} \frac{\partial g}{\partial u_j}
+ g\sqrt{1+|\nabla g|^2} \frac{1}{\sqrt{1+|\nabla g|^2}}
\frac{\partial^2 g}{\partial u_i\partial u_j }.
$$
Thus considering the Hessian matrix $\nabla^2 \Phi
=(\frac{\partial^2 \Phi}{\partial u_i \partial u_j})$
as  a tensor,
or bilinear form,
we see that it is related to the two fundamental forms by
\begin{equation}\label{nablaPhiN}
\nabla^2 \Phi = d{\bf x}\cdot d{\bf x}
+ |N_u| d{\bf x}\cdot d{\bf n}
=A + |N_u| B.
\end{equation}
We finally notice that $\Phi$ is related to the foot point map $p$ by
\begin{equation}
\nabla \Phi =p.
\label{Phip}\end{equation}
More generally, for any $a\in\R^{n-1}$ we have
\begin{equation}
\nabla \Phi_{(a,0)} (u) = p(u) - a.
\label{Phipa}\end{equation}


\subsection{The Poincar\'e metric}

The Poincar\'e metric in $(\R^n)^+$ is given by
$$
ds^2 = \frac{1}{v^2} (\sum_{j=1}^{n-1} du_j^2 + dv^2)
$$
The geodesics with respect to the Poincar\'e metric are the vertical
straight lines
$$
\begin{cases}
u= {\rm constant}, \\
v>0
\end{cases}
$$
together with all vertical semicircles with centers on $\R^{n-1}$,
namely all curves of the form
\begin{equation}\label{circle}
\begin{cases}
u \in L, \, v>0, \\
|u-a|^2 + v^2 = r^2,
\end{cases}
\end{equation}
where $L$ is a straight line in $\R^{n-1}$, $a\in L$ and $r>0$.

We consider now the variable transformation
$$
T: (u,v) \mapsto (s,t),
$$
where $s\in \R^{n-1}$, $t\in \R$, defined by
$$
\begin{cases}
s=u, \\
t=\frac{1}{2}(|u|^2 + v^2)
\end{cases}
$$
and, conversely,
\begin{equation}\label{subst}
\begin{cases}
u=s, \\
v=\sqrt{ 2t -|s|^2}.
\end{cases}
\end{equation}
Thus $T$ is one-to-one and takes $(\R^n)^+$ in the $(u,v)$-space onto the epiparabola
$$
W= \{(s,t)\in \R^n:  t > \frac{1}{2} |s|^2 \}
$$
in the $(s,t)$-space.
By the definition of $\Phi$, $T$ maps $\Omega^+$ onto the set
$$
V=\{(s,t)\in W: s\in D, \, t<\Phi (s)\},
$$
and it maps the graph of $g$ onto the graph of $\Phi$. If $\Phi$
is extended to all $\R^{n-1}$ by setting $\Phi (u)=\frac{1}{2}|u|^2$
outside $D$ (this corresponds to extending $g$ by zero outside $D$)
then $W\setminus V$ is the epigraph of $\Phi$.

\begin{lemma}\label{poincare}
The above map
$$
T:(\R^n)^+ \to W
$$
gives a one-to-one correspondence between the set of geodesics with respect to the
Poincar\'e metric in $(\R^n)^+$ and the set of straight lines in $W$.
\end{lemma}

\begin{proof}
Let $\gamma$ be
a geodesic in $(\R^n)^+$. If $\gamma$ is a vertical line,
then also $T(\gamma)$ is a vertical line in $W$, hence a straight
line. If $\gamma$ is of the form (\ref{circle}), then just performing
the substitution (\ref{subst})  we see that $T(\gamma)$ becomes
\begin{equation*}
\begin{cases}
s \in L, \\
t= s\cdot a + \frac{1}{2} (r^2 - |a|^2),
\end{cases}
\end{equation*}
hence it is a straight line. And the arguments can easily be run in
the other direction: every straight line in the $(s,t)$-space
which has a nonempty intersection with $W$ is of the form $T(\gamma)$
for some geodesic $\gamma$ in $(\R^n)^+$.
\end{proof}

In the new coordinates $(s_1,\dots, s_{n-1}, t)$ the
Poincar\'e metric takes the form
$$
ds^2 =\frac{1}{(2t-|s|^2)^2} \big( (2t-|s|^2)
\sum_{j=1}^{n-1}(ds_j)^2 +(\sum_{j=1}^{n-1}s_j ds_j)^2-
$$
$$
- \sum_{j=1}^{n-1}s_j (ds_j\otimes dt + dt\otimes ds_j) +
(dt)^2\big),
$$
as an easy calculation shows. Hence this is a metric in $W$ which is
complete and has constant negative curvature, and whose geodesics
are Euclidean straight lines.


\section{Equivalent criteria for the inner ball condition}\label{sec:innerball}

Below we state equivalent criteria for the domain $\Omega$ defined by
(\ref{Omega}) to satisfy the
inner ball condition for balls centred on the symmetry plane.
The definition of the inner ball condition can be taken to be
statement (v) in the theorem.

\begin{theorem}\label{theorem1}
With assumptions and notations as in Section~\ref{sec:diffgeom}
the following statements are equivalent.

\begin{enumerate}
\item[(i)]
The restriction of $\Phi$ to any convex subdomain of $D$ is
convex. (Note that $D$ is not required to be convex itself.)
Equivalent formulations:  the matrix of second derivatives
is positive semidefinite: $\nabla^2 \Phi \geq 0$ in $D$;
the extension of $\Phi$ to $\R^{n-1}$ is convex;
$W\setminus V$ is convex as a set.

\item[(ii)]
For every $c=(a,b)$ with $a\in \R^{n-1}$, $b>0$, the function
$\Phi_c$
is convex (alternatively: strictly convex) when restricted to any
convex subdomain of the set of $u\in D$ for which $g(u)>b$.

\item[(iii)]
The restriction of the foot point map
$p$ to any convex subdomain of $D$ is  monotone, i.e.
$$
(p(u)-p(u'))\cdot (u-u') \geq 0
$$
for all $u,u'$ in the subdomain. (Here the dot denotes the
scalar product in $\R^{n-1}$.) Equivalently, the matrix
$(\frac{\partial p_i}{\partial u_j}
+\frac{\partial p_j}{\partial u_i})$
is positive semidefinite.

\item[(iv)]
$$
\Omega =\bigcup_{u\in D} B(p(u), |N_u|).
$$

\item[(v)]
There exist radii $r=r(u) > 0$ such that
$$
\Omega = \bigcup_{u\in D} B(u,r(u)).
$$

\item[(vi)]
$$
N_{u}\cap N_{u'} = \emptyset\quad {\rm for\,\,}
u\ne u' \quad (u,u'\in D).
$$

\item[(vii)]
The principal curvatures at any point
$(u,g(u))$ of  $(\partial\Omega)^+$ are
$\geq  -\frac{1}{|N_u|}$.

\item[(viii)]
Every point in $\Omega^+$ has a unique closest neighbor
on $(\partial\Omega)^+$.

\item[(ix)]
Every point on $(\partial\Omega)^+$ is a closest point on
$\partial\Omega$ for some point in $D$.

\item[(x)]
$(\R^n)^+ \setminus \Omega$ is convex with respect to the Poincar\'e
metric in $(\R^n)^+$.
\end{enumerate}

\end{theorem}

\begin{proof}

\medskip\noindent
(i)$\Rightarrow$(ii): Fix $b>0$ and $u\in D$.
For any $\boldsymbol{\xi}\in\R^{n-1}$ with $|\boldsymbol{\xi}|=1$, let
$$
\alpha =\sum_{i,j=1}^{n-1}
\frac{\partial^2g}{\partial u_i\partial u_j}\xi_i\xi_j.
$$
If $\alpha \geq -\frac{1}{2(g(u)-b)}$ then it follows from
\eqref{nabla2Phi} that
$$
\sum_{i,j=1}^{n-1}
\frac{\partial^2\Phi_c}{\partial u_i\partial u_j}\xi_i\xi_j
\geq \frac{1}{2},
$$
while if $\alpha < -\frac{1}{2(g(u)-b)}$ equation
\eqref{nabla3Phi} shows that
$$
\sum_{i,j=1}^{n-1}
\frac{\partial^2\Phi_c}{\partial u_i\partial u_j}\xi_i\xi_j
\geq \frac{b}{2(g(u)-b)}.
$$
From these two inequalities the strict convexity of $\Phi_c$
follows. (We even get a uniform lower bound of $\nabla^2\Phi_c$
in $\{u\in D: g(u)>b\}$.)

\medskip\noindent
(ii)$\Rightarrow$(i): Just let $b\to 0$ in \eqref{nabla3Phi}.

\medskip\noindent
(i)$\Leftrightarrow$(iii): This follows from \eqref{Phip},
which also shows that the matrix
$(\frac{\partial p_i}{\partial u_j})$ is symmetric itself
($\frac{\partial p_i}{\partial u_j}
=\frac{\partial^2\Phi}{\partial u_i \partial u_j}$).

\medskip\noindent
(i)$\Leftrightarrow$(vii): As remarked in Subsection~\ref{sec:fundamentalforms} the smallest
principal curvature of $(\partial\Omega)^+$ coincides with the minimum
value of the quotient $B/A$ between the two fundamental forms.
Therefore it is immediate from \eqref{nablaPhiN} that (vii)
is equivalent to $\nabla^2\Phi \geq 0$, i.e., to (i).

\medskip\noindent
(i)$\Rightarrow$ (iv): By \eqref{OmegaB} we only need to show that
\begin{equation}
B(p(u), |N_u|)\subset \Omega.
\label{BOmega}\end{equation}
for all $u\in D$.

Fix any $u_0\in D$. Then taking $a=p(u_0)$ in \eqref{Phipa}
we see that the map $u\mapsto\Phi_{(p(u_0),0)}(u)$ has a stationary
point at $u=u_0$. In view of the interpretation of
$\Phi_{(p(u_0),0)}$ as half of the squared distance
from $(p(u_0),0)$ to $(u,g(u))$ it follows that
this stationary point is a
(global) minimum if and only if \eqref{BOmega} holds for
$u=u_0$. But when $\Phi$ (equivalently $\Phi_{(p(u_0),0)}$)
is convex then every stationary point of $\Phi_{(p(u_0),0)}$
is a global minimum.
Hence \eqref{BOmega} holds for $u=u_0$.

\medskip\noindent
(iv)$\Rightarrow$(v): If the representation in (iv) holds then
we get a representation as in (v) by adding small balls
$B(u,r(u))\subset \Omega$ for those $u\in D$ which are not in the
range of $p$. (Clearly $p$ maps $D$ into $D$ when (iv) holds,
but it need not be onto.)

\medskip\noindent
(v)$\Rightarrow$(iv): Let $(u,g(u))\in(\partial\Omega)^+$.
Then, if (v) holds,  there exist points
$a_j \in D$ and $x_j \in B(a_j, r(a_j))$ such that
$x_j\to (u,g(u))$ as $j \to\infty$.
The smoothness of $(\partial\Omega)^+$ and the inclusions
$B(a_j, r(a_j))\subset \Omega$ imply that $a_j\to p(u)$
and $r(a_j)\to |N_u|$ and we conclude that
$B(p(u), |N_u|)\subset \Omega$. Now (iv) follows from
\eqref{OmegaB}.

\medskip\noindent
(iv)$\Leftrightarrow$(ix): Note that
$(u,g(u))\in(\partial\Omega)^+$ is a
closest point of $a\in D$ if and only if
$a=p(u)$ and $B(p(u), |N_u|)\subset\Omega$.
Thus $a$ is determined by $(u, g(u))$ and it follows immediately
(in view also of \eqref{OmegaB}) that all points on $(\partial\Omega)^+$
are such closest points if and only if the inner ball condition in
(iv) holds.

\medskip\noindent
(iv)$\Rightarrow$(vi):
Assume that (vi) fails, so that there exists
a point $x\in N_{u_1}\cap N_{u_2}$ for some
$u_1,u_2 \in D$, $u_1\ne u_2$. Without loss
of generality $|(u_1,g(u_1)) -x|\leq |(u_2,g(u_2)) -x|$. Then
$(u_1,g(u_1)) \in \overline{B(x,|(u_2,g(u_2)) -x|)}$.
Since $x\in N_{u_2}$ we have
$\overline{B(x, |(u_2,g(u_2))-x|)}
\subset B(p(u_2), |N_{u_2}|)\cup\{(u_2,g(u_2))\}$.
It follows that $(u_1,g(u_1))\in B(p(u_2), |N_{u_2}|)$.
But $(u_1,g(u_1))\notin \Omega$, hence (iv) does not hold.

\medskip\noindent
(vi)$\Rightarrow$(viii):
If $x\in \Omega^+$ has two closest
neighbours $(u_1, g(u_1)), (u_2, g(u_2)) \in (\partial\Omega)^+$ then
$x\in N_{u_1}\cap N_{u_2}$.

\medskip\noindent
(viii)$\Rightarrow$(ii):
This conclusion is somewhat related to \cite{Hormander 94}, Theorem 2.1.30
(attributed to Motzkin), but we give an independent proof.

Assume that (ii) fails and we shall produce a point
$c\in\Omega^+$ with at least two closest neighbors on
$(\partial\Omega)^+$.
We may assume that $\Phi_{(0,b)}$ is not convex for some $b>0$ in a
convex subdomain of $D_b=\{ u\in D: g(u)>b\}$. We extend $\Phi_{(0,b)}$ to
all $\R^{n-1}$ by setting
\begin{align*}
\tilde{\Phi}_{(0,b)} &=
\begin{cases}
\Phi_{(0,b)}    & \quad   {\rm for \,\,} u \in D_b \\
\frac{1}{2} |u|^2 & \quad  {\rm for \,\,} u \in \R^{n-1}\setminus D_b
\end{cases} \\
&=\frac{1}{2} (|u|^2 + |(g(u) -b)_+|^2 ).
\end{align*}
In the latter expression the plus subscript denotes positive
part and $g$ is assumed to be extended by zero outside $D$.

The function $\tilde{\Phi}_{(0,b)}$ is easily seen to be
continuously differentiable and by assumption it is not convex.
Let $\Psi$ be the convex envelope of  $\tilde{\Phi}_{(0,b)}$,
which can be defined as
\begin{equation}
\Psi (u) = \sup \{ L(u): L \,\,{\rm is \,\, affine \,\, and\, }\leq
\tilde{\Phi}_{(0,b)} \}.
\label{Psi}\end{equation}
Equivalently, $\Psi$ is the function whose epigraph,
${\rm epi\,}\Psi =\{(u,v)\in\R^n: v\geq \Psi (u)\}$,
is the closed convex hull of the epigraph of $\tilde{\Phi}_{(0,b)}$.
In the present case, because of the strict convexity of
$\tilde{\Phi}_{(0,b)}$ outside a compact set, the convex hull
will be closed right away  (before taking the closure). Thus
\begin{equation}
{\rm epi\,} \Psi ={\rm conv\,}({\rm epi\,} \tilde{\Phi}_{(0,b)}),
\label{conv}\end{equation}
${\rm conv}$ denoting ``convex hull''.

Since $\tilde{\Phi}_{(0,b)}$ is not convex there exists
$u_0\in \R^{n-1}$ such that
$$
\Psi (u_0) < \tilde{\Phi}_{(0,b)}(u_0).
$$
It is easy to see that the supremum in \eqref{Psi} is attained
for each fixed $u$. Hence
$$
\Psi (u_0) = L(u_0)
$$
for some affine $L$ satisfying
\begin{equation}
L(u) \leq \tilde{\Phi}_{(0,b)}(u) \quad{\rm for\,\, all\,\,} u.
\label{LPhi}\end{equation}

By \eqref{conv}
the point $(u_0,L(u_0))=(u_0,\Psi (u_0))\in {\rm epi\,} \Psi$
can be written as a finite convex combination of points in
the epigraph of $\tilde{\Phi}_{(0,b)}$.
Since this convex combination must take place within
the graph of $L$, $u_0$ is a convex combination of finitely many
points $u$ (at least two are needed) for which
$L(u)=\tilde{\Phi}_{(0,b)}(u)$. It follows that equality is attained in
\eqref{LPhi} for at least two different $u$.

Now write
$$
L(u) =a\cdot u + k,
$$
where $a\in \R^{n-1}$, $k\in \R$ and the dot denotes scalar product
in $\R^{n-1}$. Then \eqref{LPhi} can be written
\begin{equation}
|u-a|^2 + |(g(u) -b)_+|^2 \geq 2k + |a|^2,
\label{ug}\end{equation}
where equality is attained for at least two different $u\in \R^{n-1}$.

Since the left member in \eqref{ug} is $\geq 0$ with equality in at
most one point (namely $u=a$) it follows that
\begin{equation}
2k+|a|^2 >0.
\label{pos}\end{equation}
This positive minimum value can not be attained when
$g(u)\leq b$ because for those $u$ the derivative of the
left member in \eqref{ug} is $2(u-a)\ne 0$. Note that
$u=a$ does not satisfy \eqref{ug} when $g(u)\leq b$
because of \eqref{pos}.

Thus $g(u)>b$ whenever equality holds in \eqref{ug}. Clearly also
$g(a)>b$ by \eqref{ug} and \eqref{pos}, showing that $a\in D_b$. It now
follows that
$$
|u-a|^2 + |g(u) -b|^2 \geq 2k + |a|^2
$$
for all $u\in D$ with equality for at least two different
$u$ satisfying $g(u)>b$.

In conclusion we have produced a point $c=(a,b)\in \Omega^+$
having at least two closest neighbors on $(\partial\Omega)^+$.

\medskip\noindent
(i)$\Leftrightarrow$(x): That $(\R^n)^+\setminus \Omega$ is convex
with respect to the Poincar\'e metric means by definition that
every Poincar\'e geodesic $\gamma$ in $(\R^n)^+$ meets $\Omega$
at most once (i.e., $\gamma\setminus \Omega$ has at most one component).
By Lemma~\ref{poincare} this occurs if and only if the straight line
$T(\gamma)$ in $W$ meets $V$ at most once. This is the same as saying that
$W\setminus V$ is convex as a set, which is easily seen
to be equivalent to the convexity of $\Phi$
($W\setminus V$ is the epigraph of the extension of $\Phi$
obtained by setting $g=0$ outside $D$).
This proves (i)$\Leftrightarrow$(x).

\end{proof}


\section{Boundaries of quadrature domains}\label{sec:boundaries}

\subsection{Notations}

In the rest of the paper we shall write points in $\R^n= \R^{n-1}\times \R$ typically
as $x=(x',x_n)$, with $x'=(x_1,\dots, x_{n-1})\in\R^{n-1}$. The letter $u$ will be used to
denote a specific function, see below.

Let $\mu$ be a finite positive (and not identically zero) Borel measure with compact
support in the hyperplane $\R^{n-1}=\{x\in \R^n: x_n=0\}$. Then there exists a unique
quadrature open set $\Omega\subset \R^n$ for subharmonic functions with respect to $\mu$,
see \cite{Sakai 82}, \cite{Gustafsson-Sakai 94}. Thus $\mu(\R^n\setminus\Omega)=0$ and
$$
\int_\Omega h \,d\mu\leq \int_\Omega h\,dm
$$
for all integrable (with respect to Lebesgue measure $m$) subharmonic functions $h$ in $\Omega$.
By uniqueness $\Omega$ is symmetric with respect to the hyperplane $\R^{n-1}$.
We shall assume that $\Omega$ is connected (i.e., a quadrature domain)
and discuss the curvature of $(\partial\Omega)^+$. In the nonconnected case the
discussions will simply apply to each component of $\Omega$.

Let
$$
u=U^\mu -U^\Omega
$$
be the difference between the Newtonian potentials of $\mu$ and $\Omega$,
the latter considered as a body of density one. Then
\begin{equation}\label{Deltau}
\begin{cases}
u\geq 0 \quad \text{in } \R^n,\\
u=0 \quad \text{outside } \Omega,\\
\Delta u = \chi_\Omega -\mu \quad \text{in } \R^n.
\end{cases}
\end{equation}
From \eqref{Deltau} it follows that $u\in C^1((\R^n)^+)$, hence
$\nabla u=0$ on $(\R^n)^+\setminus \Omega$ (since $u$ attains its minimum value
there).

For proving the inner ball property it is enough to consider the case that $\mu$ has a
continuous density on $\R^{n-1}$:
\begin{equation}\label{dmufdx}
d\mu = fdx_1\dots dx_{n-1},
\end{equation}
with $f=f(x')>0$ on $D=\Omega\cap\R^{n-1}$, $f=0$ on $\R^{n-1}\setminus D$
and $f$ continuous on $\R^{n-1}$ (see Subsection~\ref{sec:gammasigma}
for a motivation of this). In that case, $u$ is continuous across
$\R^{n-1}$ while $\frac{\partial u}{\partial x_n}$ makes a jump of size
$f$ across $\R^{n-1}$ (this is the meaning of the distributional Laplacian
in \eqref{Deltau}).
In summary, assuming \eqref{dmufdx}  $u$ satisfies in $\Omega^+$ the overdetermined system
\begin{equation}\label{overdetermined}
\begin{cases}
\Delta u =1 \quad \text{in }\quad \Omega^+,\\
u=|\nabla u|=0 \quad \text{on }\quad (\partial\Omega)^+,\\
-2 \frac{\partial u}{\partial x_n} = f \quad \text{on} \quad D, 
\end{cases}
\end{equation}
where in the last equation $\frac{\partial u}{\partial x_n}$
should be understood as boundary values from the upper half space.
If $\mu$ is not of the form (\ref{dmufdx}), the system (\ref{overdetermined})
still holds with the last equation appropriately reformulated.

Returning to the general case, the components of $\nabla u$ are harmonic functions in $\Omega^+$.
By applying the maximum principle to $\frac{\partial u}{\partial x_n}$ it follows that
$$
u>0 \quad \text{in } \Omega^+,
$$
\begin{equation}\label{dudxnneg}
\frac{\partial u}{\partial x_n} <0 \quad \text{in } \Omega^+.
\end{equation}
Thus $u$ is strictly decreasing with respect to $x_n$, showing that
$(\partial\Omega)^+$ is the graph of a function, i.e., is of the
form $x_n =g(x')$ for some positive function $g$ defined in
$D=\Omega\cap\R^{n-1}$. With a more detailed
analysis one can show \cite{Friedman 82}, \cite{Caffarelli 98}, \cite{Gustafsson-Sakai 94}
that the function $g$ is actually
real analytic and tends to zero at $\partial D$.
It follows that we are in the setting of Sections~\ref{sec:diffgeom} and \ref{sec:innerball}.

The aim of the forthcoming analysis is to approach the following
conjecture.

\begin{conjecture}\label{conj:main}
Any subharmonic quadrature domain $\Omega$ as above satisfies
the equivalent conditions in Theorem~\ref{theorem1}.
\end{conjecture}
As have already been remarked, Conjecture~\ref{conj:main} has been
settled in the case $n=2$. Two completely different proofs appear in
\cite{Gustafsson-Sakai 03} and \cite{Gustafsson-Sakai 04}. A third proof, perhaps
somewhat related to the one in \cite{Gustafsson-Sakai 04}, will be given in the
present paper (Corollary~\ref{cor:main}).


\subsection{The function $\rho$ and the $2$-form $\omega$}

Let $(r, \theta)$ be spherical coordinates in $\R^n$ in the sense that
$r =|\x|=\sqrt{x_1^2+\dots+x_n^2}$, $\theta=\frac{\x}{|\x|}\in S^{n-1}$
(if $\x\ne 0$).
Then $r \frac{\partial }{\partial r}=\x\cdot \mathbf{\nabla}
=x_1 \frac{\partial }{\partial x_1} +\dots+ x_n\frac{\partial }{\partial x_n}$.
In order to analyse the curvature of $(\partial\Omega)^+$ we
introduce a function $\rho$ and a 2-form $\omega$ according to

\begin{align}
\rho &= \frac{1}{2} r^2  - r \frac{\partial u}{\partial r} - (n-2)u, \label{rho}\\
\omega &= rdr \wedge du = d(\frac{1}{2}r^2 du).  \notag
\end{align}
Thus,
$$
\omega =\sum_{i<j} \omega_{ij} dx_i \wedge dx_j =\sum_{i,j=1}^n \omega_{ij} dx_i \otimes dx_j
$$
($dx_i \wedge dx_j=dx_i \otimes dx_j -dx_j \otimes dx_i$) with
\begin{equation}\label{omegaxdudx}
\omega_{ij}= x_i\frac{\partial u}{\partial x_j}-x_j\frac{\partial u}{\partial x_i}.
\end{equation}
Note that the restriction of $\rho$
to ${(\partial\Omega)^+}$ coincides with
the function $\Phi$ in Subsection~\ref{Phi}
when the latter is regarded as a function of $x'=(x_1,\dots, x_{n-1})$: with $(\partial\Omega)^+$
given by $x_n=g(x')$ we have
$$
\Phi (x')= \rho (x', g(x')).
$$

Recall that the Hodge star operator \cite{Frankel 97} is
defined on basic forms
$$
\alpha =dx_{i_1}\wedge\dots\wedge dx_{i_p}
$$
as
$$
\ast\alpha = \pm dx_{i_{p+1}}\wedge\dots\wedge dx_{i_n},
$$
where the sign is chosen so that
$$
\alpha\wedge\ast\alpha = dx_1\wedge\dots\wedge dx_n = dx
\qquad \text{(the volume form)}
$$
and $(i_1,\dots ,i_n)$ denotes any permutation of $(1,\dots,n)$.
Interior multiplication of a $p$-form $\alpha$ with a vector field
${\boldsymbol{\xi}}$ is the $(p-1)$-form obtained by letting ${\boldsymbol{\xi}}$
occupy the first position when $\alpha$ is viewed as an antisymmetric
operator acting on $p$ vector fields:
$$
(i({\boldsymbol{\xi}})\alpha )({\boldsymbol{\xi}_2},\dots, {\boldsymbol{\xi}_p})
=\alpha ({\boldsymbol{\xi}},{\boldsymbol{\xi}_2},\dots, {\boldsymbol{\xi}_p}).
$$
We shall also need the Lie derivative
$L_{\boldsymbol{\xi}}$ with respect
to a vector field $\boldsymbol{\xi}$, and recall that its action on any
$p$-form $\alpha$ is
$$
L_{\boldsymbol{\xi}}(\alpha )=i(\boldsymbol{\xi} ) d\alpha + d( i(\boldsymbol{\xi}) \alpha),
$$
and that its action on a vector field $\boldsymbol{\eta}$ is the commutator
$$
L_{\boldsymbol{\xi}} \boldsymbol{\eta} =[\boldsymbol{\xi}, \boldsymbol{\eta}],
$$
when all vector fields are considered as differential operators
($\boldsymbol{\xi}$ corresponds to the directional derivative $\boldsymbol{\xi}\cdot \nabla$ etc.).

In our Euclidean setting of differential geometry there are natural one-to-one correspondences
between vector fields, $1$-forms and $(n-1)$-forms.
The basis elements correspond to each other according to
$$
{\bf e}_i \leftrightarrow dx_i \leftrightarrow  \ast dx_i,
$$
and the forms/vectors simply keep their coefficients with respect to the
above bases under the identifications. This means for example that
if $\mathbf{a}, \mathbf{b}$ are vectors, $\alpha$ a $1$-form
and $\beta$ an $(n-1)$-form, then
$$
\mathbf{a}\leftrightarrow \alpha \quad \text{if and only if}\quad i(\mathbf{a}) dx =\ast \alpha,
$$
$$
\mathbf{b}\leftrightarrow \beta \quad \text{if and only if}\quad i(\mathbf{b}) dx =\beta.
$$
The definition of $\omega$ can be expressed in terms of interior
multiplication and the Hodge star operator as
$$
\ast\omega= -i(r\frac{\partial}{\partial r}) ({\ast du}).
$$


\subsection{The Cauchy-Riemann system}

The fundamental relationship between $\rho$ and $\omega$ is
expressed in the following Cauchy-Riemann type system.

\begin{theorem}
We have
\begin{equation}\label{CR}
\ast d\rho =d(\ast \omega) \quad {\rm in}\quad \Omega^+.
\end{equation}
Written out in components this is
\begin{equation}\label{components}
\frac{\partial\rho}{\partial x_k}
=\sum_{j=1}^n \frac{\partial\omega_{kj}}{\partial x_j} \quad (k=1,\dots ,n).
\end{equation}
\end{theorem}

\begin{proof}
It is easy to check that \eqref{CR}, when spelled out, simply
becomes \eqref{components}, so we just verify
\eqref{components}: Differentiating \eqref{omegaxdudx} we get
$$
\frac{\partial\omega_{kj}}{\partial x_j}=x_k \frac{\partial^2 u}{\partial x_j^2}
-\frac{\partial u}{\partial x_k}-x_j \frac{\partial^2 u}{\partial x_j \partial x_k}
$$
when $j\ne k$. For $j=k$ there is an additional term $\frac{\partial u}{\partial x_j}$,
which cancels the term $-\frac{\partial u}{\partial x_k}$. Summing over $j$ therefore
gives
$$
\sum_{j=1}^n \frac{\partial\omega_{kj}}{\partial x_j} = x_k \Delta u
-(n-1)\frac{\partial u}{\partial x_k} - r\frac{\partial}{\partial r}(\frac{\partial u}{\partial x_k}).
$$
On the other hand, differentiating \eqref{rho} gives
$$
\frac{\partial\rho}{\partial x_k}=x_k -(n-2)\frac{\partial u}{\partial x_k}
-  \frac{\partial }{\partial x_k} (r\frac{\partial u}{\partial r}).
$$
By a change of order of differentiating, and in view of $\Delta u =1$ in $\Omega^+$,
\eqref{components} now follows.
\end{proof}

Introducing the coexterior derivative $\delta$, defined on $p$-forms by
$$
\delta = (-1)^{n(p+1)+1}\ast d\ast,
$$
the system \eqref{CR} can be written
\begin{equation}\label{drhodeltaomega}
d\rho =\delta \omega.
\end{equation}
Trivially $\delta \rho =0$, and since $\omega$ by definition is exact,
$d\omega=0$. Recall that the Hodge Laplacian is
$$
\Delta =-(\delta d + d\delta)
$$
and that it in our Euclidean setting simply equals the ordinary Laplacian
acting on the coefficients of the forms it is applied to.
From (\ref{drhodeltaomega}) we find (using $\delta \circ\delta =0$ etc.)
\begin{corollary}\label{rhoharmonic}
$$
\begin{cases}
\delta \rho =0,\\
\delta d \rho =0,
\end{cases}
\quad
\begin{cases}
d \omega =0,\\
d\delta \omega =0,
\end{cases}
$$
in particular
$$
\Delta \rho =0,\quad
\Delta  \omega =0.
$$
\end{corollary}


\subsection{The vector field $\boldsymbol{\xi}$}

Next we introduce the vector field
\begin{align*}
\boldsymbol{\xi} =\nabla \rho &= {\bf x}-(n-2)\nabla u-\nabla ({\bf x}\cdot \mathbf{\nabla} u) \\
&={\bf x}-(n-1 + r\frac{\partial}{\partial r})\nabla u.
\end{align*}
For any function $\varphi$,
$$
i(\nabla \varphi) dx= \ast d\varphi,
$$
hence  the Cauchy-Riemann system \eqref{CR} can be written in terms
of $\boldsymbol{\xi}$ and $\omega$ as
$$
i(\boldsymbol{\xi})dx =d(\ast \omega).
$$
Immediate consequences, in view of $i(\boldsymbol{\xi}) \circ i(\boldsymbol{\xi})=0$, are

$$
i(\boldsymbol{\xi} ) d(\ast\omega) =0,
$$
\begin{equation}\label{Lie}
L_{\boldsymbol{\xi}} (\ast \omega)= d(i(\boldsymbol{\xi})(\ast\omega)).
\end{equation}
When $n=2$, $i(\boldsymbol{\xi})(\ast\omega)=0$ (because it is an $(n-3)$-form),
hence $L_{\boldsymbol{\xi}} (\ast \omega)=0$.

We proceed with a geometric interpretation of $\boldsymbol{\xi}$ on
$(\partial\Omega)^+$.

\begin{proposition}\label{projection}
At any point ${\bf x}\in (\partial\Omega)^+$,
$\boldsymbol{\xi}=\nabla \rho$
equals the orthogonal projection of the vector
${\bf x}$
onto the tangent space of $(\partial\Omega)^+$
at ${\mathbf x}$. In particular, $\boldsymbol{\xi}$ is tangent to $(\partial\Omega)^+$,
i.e., $\frac{\partial \rho}{\partial n}=0$  on $(\partial\Omega)^+$.
\end{proposition}

\begin{proof}
Since all components of $\ast\omega$ vanish on $(\partial\Omega)^+$, the $(n-1)$-form
$i(\boldsymbol{\xi})dx =d(\ast \omega)$ vanishes along $(\partial\Omega)^+$
(i.e., its integral over any piece of $(\partial\Omega)^+$ is zero).
This means exactly that $\boldsymbol{\xi}$ is tangent to $(\partial\Omega)^+$.
But $\boldsymbol{\xi}={\bf x}-\nabla ({\bf x}\cdot \nabla u -(n-2) u)$
and $\mathbf{\nabla}({\bf x}\cdot \nabla u +(n-2) u)$
is orthogonal to $(\partial\Omega)^+$ because
${\bf x}\cdot \nabla u +(n-2) u =0 $ on $(\partial\Omega)^+$.
It follows that $\boldsymbol{\xi}$ is the projection of ${\bf x}$
onto the tangent space of $(\partial\Omega)^+$.

\end{proof}

\begin{corollary}
{A point} ${\bf x} = (x',x_n)\in (\partial\Omega)^+$ {is stationary for} $\frac{1}{2}r^2$ {on} $(\partial\Omega)^+$
equivalently, $x'$ {is stationary for} $\Phi${, if and only if}
$\boldsymbol{\xi}=0$ {at} ${\bf x}$.
\end{corollary}

\begin{example}
In two dimensions, with the identification $\R^2=\C$ and using $(x, y)$ as coordinates
(and with $z=x+\I y$) we can define the Schwarz function in $\Omega^+$ by
$$
S(z) = \bar{z}-4\frac{\partial u}{\partial z}.
$$
It is holomorphic in  $\Omega^+$ and extends continuously to $(\partial \Omega)^+$ with
$$
S(z)= \bar{z} \quad \text{on} \quad (\partial \Omega)^+.
$$
In addition, by \eqref{overdetermined},
$$
\im S(x+i0) = f(x) \quad \text{on}\quad \partial(\Omega^+)\cap \R.
$$

In the two dimensional case $\ast\omega$ is a $0$-form, i.e., a function, and it is exactly the conjugate
harmonic function of $\rho$. More precisely, $\rho$, $\ast\omega$ are the real and imaginary parts of the analytic function
$\frac{1}{2} zS(z)$:
\begin{equation}\label{sigma}
\frac{1}{2} zS(z) = \frac{1}{2} r^2 -r\frac{\partial u}{\partial r}
+ \I \frac{\partial u}{\partial \varphi}
= \rho +\I \ast\omega.
\end{equation}

\end{example}

\begin{example}
In the case of three dimensions we can identify
the $2$-form $\omega$ with the vector field
$$
\boldsymbol{\omega} = {\bf x} \times {\bf \nabla} u
=  \curl (\frac{1}{2}r^2 \nabla u).
$$
The Cauchy-Riemann system  $\ast d\rho =d(\ast\omega)$ then takes the form
$$
\nabla \rho =\curl \boldsymbol{\omega}.
$$
In terms of $\boldsymbol{\xi}=\nabla \rho$ and $\boldsymbol{\omega}$ we thus have the equations
$$
\begin{cases}
\di \boldsymbol{\omega} =0,\\
\curl \boldsymbol{\xi} =0,\\
\boldsymbol{\xi} =\curl \boldsymbol{\omega}.
\end{cases}
$$
In particular $\di \boldsymbol{\xi}=0$.
Note the triple curl identity $\curl\curl\curl (\frac{1}{2} r^2\nabla u)=0$.

Notice also that the vector field $\boldsymbol{\omega}$ is everywhere tangent
to the family of spheres $|\mathbf{x}|={\rm constant}$. Thus, in spherical
coordinates $(r, \theta, \varphi)$, where $\theta$ (now) is the angle between the
vector $\mathbf{x}$ and the $x_3$-axis and $\varphi$ is the polar angle for the
projection of $\mathbf{x}$ onto the $(x_1,x_2)$-plane, on writing
$$
\boldsymbol{\omega}= \omega_r \mathbf{{e}}_r
+\omega_\theta \mathbf{{e}}_\theta
+\omega_\varphi \mathbf{{e}}_\varphi
$$
we have $\omega_r=0$. On $D=\Omega \cap \R^2$, i.e., for $\theta
=\frac{\pi}{2}$
we have $\omega_\varphi = -r\frac{\partial u}{\partial x_3} = \frac{1}{2}r{f(x')}>0$ because $f > 0$ on $D$.

For the interior multiplication,
\begin{align*}
i(\boldsymbol{\xi})(\ast\omega) = \boldsymbol{\omega}\cdot\boldsymbol{\xi}
&=(\mathbf{x}\times \nabla u)\cdot (\mathbf{x}-\nabla(\mathbf{x}\cdot \nabla u)- \nabla u) \\
&=-(\mathbf{x}\times \nabla u)\cdot \nabla (\mathbf{x}\cdot \nabla u).
\end{align*}
As to the Lie derivatives of $\boldsymbol{\omega}$ (vector), $\omega$
($1$-form), $\ast\omega$ ($2$-form) we have,
when the results are identified with vector fields,
\begin{align*}
L_{\boldsymbol{\xi}} ( \boldsymbol{\omega} )&=[\boldsymbol{\xi},\boldsymbol{\omega}]
=(\boldsymbol{\xi} \cdot\nabla) \boldsymbol{\omega} - (\boldsymbol{\omega}\cdot\nabla )\boldsymbol{\xi}, \\
L_{\boldsymbol{\xi}} (\omega) &=\text{ as a vector }= \curl (\boldsymbol{\omega} \times \boldsymbol{\xi})
=(\boldsymbol{\xi} \cdot\nabla) \boldsymbol{\omega} - (\boldsymbol{\omega}\cdot\nabla )\boldsymbol{\xi}, \\
L_{\boldsymbol{\xi}} (\ast\omega) &=\text{ as a vector }
={\nabla}(\boldsymbol{\xi}\cdot\boldsymbol{\omega})
=(\boldsymbol{\xi} \cdot\nabla) \boldsymbol{\omega} + (\boldsymbol{\omega}\cdot\nabla )\boldsymbol{\xi}.
\end{align*}

\end{example}

\bigskip

We now return to the general case.
Since, by \eqref{dudxnneg}, $\nabla u\ne 0$ in $\Omega^+$, the integral curves of $\nabla u$
constitute a smooth (even real analytic) family of curves which fill up
$\Omega^+$. The curves are directed downwards ($\nabla u\cdot {\bf{e}}_n<0$),
start immediately inside $(\partial\Omega)^+$ (recall that $\nabla u =0$ on $(\partial\Omega)^+$)
and end up at points of $\supp \mu$. Thus a subdomain $R\subset \Omega^+$ which is
bounded by integral curves of $\nabla u$ may be thought
of as a tube going from $(\partial\Omega)^+$  to $D\subset\R^{n-1}$.

\begin{proposition}
Let $R\subset\Omega^+$ be a domain such that $\partial R\cap \Omega^+$ is smooth and consists of
integral curves of $\nabla u$.
Then, assuming $d\mu =fdx_1\dots dx_{n-1}$ with $f$ continuous on $D$,
$$
m (R)
=\frac{1}{2}\mu( \partial{R}\cap \R^{n-1}),
$$
$m$ denoting Lebesgue measure.
\end{proposition}

\begin{proof}
By assumption, $\nabla u\cdot {\bf{n}}=0$ on $\partial R\cap \Omega^+$ (${\bf{n}}$ the normal unit vector on $\partial R$).
Therefore,
\begin{align*}
m(R) &= \int_R \Delta u dx
=\int_{\partial R} \nabla u\cdot {\bf {n}}d\sigma \\
&=\int_{ \partial R\cap\R^{n-1}} \nabla u\cdot {\bf {n}}d\sigma
+\int_{\partial R\cap \Omega^+} \nabla u\cdot {\bf {n}}d\sigma
+\int_{\partial R\cap (\partial\Omega)^+} \nabla u\cdot {\bf {n}}d\sigma \\
&=\int_{ \partial R\cap\R^{n-1}} (-\frac{\partial u}{\partial x_n}) dx_1\dots dx_{n-1} +0+0
=\frac{1}{2}\mu( \partial{R}\cap\R^{n-1}).
\end{align*}

\end{proof}


\subsection{The curve $\gamma$ and the map $\sigma$}\label{sec:gammasigma}

Next we shall study the set where $\omega$ vanishes:
$$
\gamma=\{x\in\Omega^+ : \omega (x)=0 \}.
$$

The vanishing of $\omega$ at a point $x$ means by definition that
all components of $\omega$ vanish. Note that, since $rdr\ne 0$ and
$du\ne 0$ in $\Omega^+$, $\omega$ vanishes if and only if $rdr$ and
$du$ are proportional, i.e., if and only if $\nabla u$ is parallel
to the vector from $x$ to the origin, and hence points towards the origin. This is
also the same as saying that the projection
$$
\nabla_{\partial B_r} u= \nabla u -(\nabla u \cdot {\bf e}_r){\bf e}_r
$$
of $\nabla u$ onto the tangent plane of the sphere $\partial B (0,r)$
vanishes. Hence
\begin{align}
\gamma &=\{x\in\Omega^+ : -\nabla u (x) \,\,\text{is parallel to the
vector } \x \} \label{gamma}\\
&=\{ x\in\Omega^+ : \nabla_{\partial B_r} u(x) =0\}. \notag
\end{align}
Recall that $\x$ denotes the vector from the origin to the point $x$.

A possible proof of Conjecture~\ref{conj:main} may be based on establishing
that $\gamma$ is a smooth curve, when nonempty. The strategy would then
be as follows. Suppose $\Omega$ does not satisfy the inner ball
condition, i.e., that the equivalent conditions in
Theorem~\ref{theorem1} do not hold. Then there exists by condition
(viii) a point $c=(a,b)\in\Omega^+$ which has at least two closest
neighbors on the boundary. Here $a\in\R^{n-1}$, $b>0$. By
``localization'' (see \cite{Gustafsson-Sakai 94}, \cite{Sakai 10})
the set $\Omega_c^+=\{x\in \Omega^+: x_n>b\}$ will still be the
upper half of a quadrature domain, namely for a measure with support on the hyperplane
$x_n=b$. Moreover, this measure will have a continuous density function $f$ on $x_n=b$,
namely given by
$$
f(x_1,\dots, x_{n-1})=-2\frac{\partial u}{\partial x_n} (x_1,\dots,x_{n-1},b)
$$
(cf. (\ref{overdetermined})). This is because the localization measure on $x_n=b$
can be obtained from $u$ by
keeping $u$ unchanged in $x_n>b$ and defining a new continuation of $u$ to $x_n<b$ by reflection
in $x_n=b$. Note by (\ref{dudxnneg}) that the above $f$ is strictly positive at points
belonging to $\Omega$.

Now we translate our system of coordinates so
that $c$ becomes the origin in the new coordinate
system,  and we also adapt the notations in general to fit the new coordinates.
In this new situation $\Omega$ is a subharmonic quadrature domain
for a measure $\mu$ in the hyperplane $x_n=0$ with
\begin{equation}\label{musmooth}
d\mu =f dx_1\dots dx_{n-1},
\end{equation}
$f$ continuous on $\R^{n-1}$, $f> 0$ on $D=\Omega\cap \R^{n-1}$, and the largest ball $B(0,r_0)$ centered at
the origin and contained in $\Omega$ touches $(\partial\Omega)^+$ in at least two points.
The set $\gamma$ is expected to meet $(\partial\Omega)^+$
at these points (cf. Proposition~\ref{prop:gamma} below), and in addition reach the origin,
thus it need to branch.
Thus if we can prove that $\gamma$ is a smooth curve we reach a contradiction.
The conclusion then is that the original
$\Omega$ will have to satisfy the inner ball condition.

Motivated by the above we assume from now on
that $\mu$ is of the form \eqref{musmooth} with $f>0$ on $D$ and continuous on $\R^{n-1}$.

\begin{proposition}
If  $0\notin {\rm conv\,}\overline{\Omega}$,
then $\gamma$ is empty.
If $0\in\Omega$ then $\gamma$ intersects every hemisphere $(\partial B (0,r))^+$
with $0<r<r_0=\inf_{x\in (\partial\Omega)^+} |\x|$.
For $r>0$ sufficiently small there is exactly one intersection point.
At the origin $\gamma$ starts out in the direction $\boldsymbol{\xi}(0)$. The
closure of $\gamma$ reaches $\Omega\cap\R^{n-1}$ only at the origin.
\end{proposition}

\begin{proof}
To prove the first statement,
assume $\supp \mu\subset \{(x_1,\dots,x_{n-1})\in\R^{n-1}: x_1< 0\}$,
for example, and consider the component
$$
\omega_{1n}=x_1\frac{\partial u}{\partial x_n}-x_n\frac{\partial u}{\partial x_1}
$$
of $\omega$. We have $\Delta \omega_{1n}=0$  in $\Omega^+$, $\omega_{1n}=0$ on
$(\partial\Omega)^+$. On $\R^{n-1}$, $\omega_{1n}=x_1\frac{\partial u}{\partial x_n}$,
which by assumption (and \eqref{overdetermined}) vanishes for $x_1\geq 0$
and is nonnegative (and not identically zero) for $x_1<0$. Hence $\omega_{1n}>0$,
and so $\omega\ne 0$, in $\Omega^+$.

To prove the second statement, consider any hemisphere $(\partial B(0,r))^+$,
$0<r<r_0$, and the tangent vector field
$\nabla_{\partial B_r} u= \nabla u -(\nabla u \cdot {\bf e}_r){\bf e}_r$ on it. On
the boundary, $\partial B (0,r)\cap \R^{n-1}$, $\nabla_{\partial B_r} u$
points into the lower hemisphere, because $\frac{\partial u}{\partial x_n}<0$.
Thus it follows from well-know index theorems for vector fields (see
\cite{Frankel 97}, for example) that
$\nabla_{\partial B_r} u$ has to vanish somewhere in $(\partial B(0,r))^+$.

Next, assume
$x,y\in\gamma\cap(\partial B(0,r))^+$ ($x\ne y$). Then
$\x=-\lambda \nabla u(x)$, $\y=-\mu \nabla u(y)$ for some $\lambda , \mu >0$
($\x$, $\y$ denote $x$ and $y$ considered as vectors).
An easy calculation, using $|\x|=|\y|$, shows that
$$
\frac{|\nabla u(x)-\nabla u(y)|}{|\x-\y|}
=\frac{1}{\lambda}\frac{|\x-\frac{\lambda}{\mu}\y|}{|\x-\y|}
\geq\frac{1}{\lambda}\inf_{\alpha>0}\frac{|\x-\alpha \y|}{|\x-\y|}
\geq\frac{1}{2\lambda}.
$$

By assumption, $\nabla u$ is smooth in $\Omega^+$ up to $\R^{n-1}$ and
$|\nabla u|\geq -\frac{\partial u}{\partial x_n}\geq c$ for small $r>0$ (for some $c>0$).
It follows that $\lambda=\frac{r}{|\nabla u(x)|}\leq\frac{r}{c}$,
hence $\lambda \to 0$ if we let $r\to 0$. Thus the difference quotients
$\frac{|\nabla u(x)-\nabla u(y)|}{|\x-\y|}$ on $(\partial B(0,r))^+$ tend
to infinity as $r\to 0$, which contradicts the smoothness of $\nabla u$
up to $\R^{n-1}$. The conclusion is that there cannot be two different
points in $\gamma\cap (\partial B(0,r))^+$ for $r>0$ sufficiently small.

Thus, for small $r>0$ there is a unique point
$x=x(r)\in \gamma\cap (\partial B(0,r))^+$.
The direction as $r\to 0$ is (writing $\x(r)=-\lambda(r)\nabla u(x(r))$)
$$
\lim_{r\to 0} \frac{\x(r)}{|\x(r)|}
=\lim_{r\to 0} \frac{-\lambda(r)(-\nabla u(x(r)))}{\lambda(r)|\nabla u(x(r))|}
=-\frac{\nabla u(0)}{|\nabla u(0)|}=\frac{\boldsymbol{\xi}(0)}{|\boldsymbol{\xi}(0)|}.
$$

The last statement of the proposition follows directly from the fact that
$\omega_{jn} = x_j\frac{\partial u}{\partial x_n}= -\frac{1}{2}x_j f(x')\ne0$
on $(\Omega\cap\R^{n-1}) \setminus \{0\}$.

\end{proof}

From the definition \eqref{omegaxdudx} of $\omega_{kj}$ we see that
$$
\omega_{kj}=\frac{1}{x_n} (x_j \omega_{kn} -x_k \omega_{jn}),
$$
hence $\omega$ is algebraically determined by the $n-1$ components
$\omega_{1n},\dots ,\omega_{n-1,n}$. It follows that all information
of $\rho$ and $\omega$ is contained in the map
$$
\sigma: \Omega^+ \to \R^{n}, \quad x\mapsto (\omega_{1n} (x),\dots ,\omega_{n-1,n}(x), \rho(x)).
$$
In the case of two dimensions this $\sigma=(\omega_{12},\rho)$ is, with a switch of coordinates,
the same as the function $\frac{1}{2}zS(z)$ in \eqref{sigma}:
$\frac{1}{2}zS(z)=\rho + \I \ast\omega=\rho + \I \omega_{12}$.
In particular we conclude from the above that
\begin{align*}
\gamma &=\{x\in\Omega^+ : \omega_{1n}(x)=\dots = \omega_{n-1,n}(x)=0 \} \\
&=\sigma^{-1}(\text{the $n$:th coordinate axis}).
\end{align*}

Let $J_\sigma$ denote the Jacobian matrix of $\sigma$:
$$
J_\sigma=
\left(
\begin{array}{llll}
\frac{\partial \omega_{1n}}{\partial x_1} & \frac{\partial \omega_{1n}}{\partial x_2}
& \dots  & \frac{\partial \omega_{1n}}{\partial x_n}\\ \\
\frac{\partial \omega_{2n}}{\partial x_1} & \frac{\partial \omega_{2n}}{\partial x_2}
& \dots & \frac{\partial \omega_{2n}}{\partial x_n}\\ \\
\dots  & \dots  & \dots   &  \dots \\ \\
\frac{\partial \omega_{n-1,n}}{\partial x_1}  & \frac{\partial \omega_{n-1,n}}{\partial x_2}
& \dots & \frac{\partial \omega_{n-1,n}}{\partial x_n}\\ \\
\frac{\partial \rho}{\partial x_1} & \frac{\partial \rho}{\partial x_2}
& \dots & \frac{\partial \rho}{\partial x_n}
\end{array}
\right).
$$
Equation \eqref{components} with $k=n$ says that
$$
{\rm tr\,} J_\sigma =0.
$$
Equivalently, if $\sigma$ considered as a vector field,
$$
{\rm div\,} \sigma=0.
$$
Note also, by Corollary~\ref{rhoharmonic}, that $\sigma$ is a harmonic map
(each component of $\sigma$ is a harmonic function).

Clearly $\sigma$ extends continuously to $\overline{\Omega^+}$. On $(\partial\Omega)^+$,
\begin{equation}\label{sigma00}
\sigma =(0,\dots, 0,\frac{1}{2}r^2),
\end{equation}
hence $(\partial\Omega)^+$, of dimension $n-1$, is mapped by $\sigma$ onto a set of dimension at most one.
It follows that ${\rm rank\,} J_\sigma \leq 2$ at all points of $(\partial\Omega)^+$
(since $\sigma$ loses at least $n-2$ dimensions). To be more precise, for any vector
$\boldsymbol{\eta}$,
$$
J_\sigma \boldsymbol{\eta} =
\left(
\begin{array}{l}
\boldsymbol{\eta}\cdot \nabla \omega_{1n}\\ \\
\boldsymbol{\eta}\cdot \nabla \omega_{2n}\\ \\
   \dots   \\ \\
\boldsymbol{\eta}\cdot \nabla \omega_{n-1,n}\\ \\
\boldsymbol{\eta}\cdot \nabla \rho
\end{array}
\right),
$$
and writing, at a given point on $(\partial\Omega)^+$, $\boldsymbol{\eta}$ as
$$
\boldsymbol{\eta} = \boldsymbol{\eta}_t + \eta_n \, {\bf {n}}
$$
with $\boldsymbol{\eta}_t$ tangent to $(\partial\Omega)^+$ and
$\eta_n\in \R$ gives (since $\omega_{jn}=0$ on $(\partial\Omega)^+$),
$$
J_\sigma \boldsymbol{\eta} =
(\boldsymbol{\eta}_t\cdot \nabla \rho)\,
{\bf {e}}_n +\eta_n \,J_\sigma \,({\bf {n}})
=(\boldsymbol{\eta}_t\cdot {\bf x})\,
{\bf {e}}_n +\eta_n \,J_\sigma \,({\bf {n}}),
$$
where ${\bf e}_n$ is the $n$:th unit vector
(for the last member Proposition~\ref{projection} was used). Thus
the range of $J_\sigma$ is contained in
the (at most) two dimensional vector space spanned by
${\bf {e}}_n$ and $J_\sigma ({\bf {n}})$.

At points on $(\partial\Omega)^+$ where $\frac{1}{2}r^2$ is stationary,
i.e., $\boldsymbol{\xi}=0$, we have $\boldsymbol{\eta}_t\cdot \nabla \rho=0$,
hence ${\rm rank\,}J_\sigma \leq 1$. Since the trace of $J_\sigma$ is zero
it follows that all eigenvalues of $J_\sigma$ are zero. In two
dimensions it even follows that $J_\sigma=0$, because $J_\sigma$ is
symmetric when $n=2$.

Returning to the mapping properties of $\sigma$, we see from \eqref{sigma00} that $\sigma$
maps $(\partial\Omega)^+$ onto the segment $\frac{1}{2}r_0^2< \sigma_n<\frac{1}{2}r_1^2$ of the $\sigma_n$-axis,
possibly except for endpoints,
where
$$
r_0 =\inf_{(\partial\Omega)^+} r,
\qquad r_1 =\sup_{(\partial\Omega)^+} r.
$$
On $D=\Omega\cap\R^{n-1}$ we have
\begin{align}
\sigma &=(x_1 \frac{\partial u}{\partial x_n},\dots, x_{n-1} \frac{\partial u}{\partial x_n},
\frac{1}{2}r^2-r \frac{\partial u}{\partial r}- (n-2)u) \notag \\
&=(-\frac{x_1}{2} f(x'), \dots , -\frac{x_{n-1}}{2} f(x'),
\frac{1}{2}r^2-r \frac{\partial u}{\partial r}- (n-2)u).\label{opposite}
\end{align}
Thus, since we have assumed that $f>0$ on $D$, $\sigma(D)$ meets the $\sigma_n$-axis only if $0\in\Omega$.
The meeting point is
$\sigma (0)= (0,\dots,0, -(n-2)u(0))$, hence has $\sigma_n(0)\leq 0$.

In summary, $\sigma$ maps the boundary $\partial(\Omega^+)$ of $\Omega^+$ as follows:
$0\in\partial (\Omega^+)$ is mapped onto a point $(0,\dots,0,\sigma_n)$ with $\sigma_n=-(n-2)u(0)\leq 0$,
$\partial(\Omega^+)\cap (\R^n)^+= (\partial \Omega)^+$ is mapped to points on the $\sigma_n$-axis,
with $\sigma_n \geq \frac{1}{2} r_0^2 >0$, and for the remaining points $x\in D\setminus \{0\}$ of
$\partial(\Omega^+)$, $\sigma (x)$ does not meet the $\sigma_n$-axis.
More precisely, $x$ and $\sigma (x)$ are by \eqref{opposite} located
on opposite sides of the $\sigma_n$-axis.
It follows from all this that $\partial(\Omega^+)$ encloses the
interval $[-(n-2)u(0),\frac{1}{2}r_0^2]$
on the $\sigma_n$-axis $(-1)^{n-1}$ times in the sense of degree theory. Thus

\begin{theorem}\label{thm:main}
Assume $0\in\Omega$ and that $\mu$ is of the form \eqref{musmooth} with $f>0$ on $D$.
Then the map  $\sigma$ has mapping degree (Brouwer degree)
$(-1)^{n-1}$ with respect to each point $(0,\dots, 0, t)$ with $-(n-2)u(0)<t<\frac{r_0^2}{2}$.
\end{theorem}

A conclusion is that each value $(0,\dots, 0, t)$, $0<t<\frac{r_0^2}{2}$, is
attained at least once, but it can also be attained several times with
signatures (signs of the Jacobi determinant) adding up to $(-1)^{n-1}$.

In the case of two dimensions a stronger conclusion is possible because
the Jacobi determinant has constant sign: $\det J_\sigma \leq 0$.
Switching the order between $\omega$ and $\rho$ one gets the analytic
function $\frac{1}{2}zS(z)=\rho + \I \ast\omega$
(which we still denote $\sigma(z)$), with nonnegative Jacobi determinant,
and all arguments become more familiar, e.g., the mapping degree can be identified
with the argument
variation divided by $2\pi$. We repeat everything in this case, with stronger
assertions.

\begin{corollary}\label{cor:main}
With assumptions as in the theorem we have, in the case $n=2$,
that if $0\in\Omega$ then there is exactly one point on $(\partial\Omega)^+$
of shortest distance to the origin and
$\gamma$ is an analytic smooth arc from origin to that point. In particular, the equivalent conditions
in Theorem~\ref{theorem1} hold.
In addition, $\gamma$ is an integral curve of $\boldsymbol{\xi}$.
\end{corollary}

\begin{proof}
Let $r_0>0$ be the distance from the origin to $(\partial\Omega)^+$. We apply
the argument principle to the analytic function
${\sigma}(z)=\frac{1}{2}zS(z)=\rho + \I \ast\omega$ in $\Omega^+$.
We have $\sigma (0)=0$, $\im \sigma (x)>0$ for $x<0$, $\im \sigma (x)<0$ for $x>0$.
For $z\in (\partial\Omega)^+$, $\sigma (z)=\frac{1}{2}|z|^2$, hence
$\sigma((\partial\Omega)^+)\subset [\frac{1}{2}r_0^2,\infty)$.

It follows that as $z$ runs through $\partial(\Omega^+)$, then ${\sigma} (z)$ encircles each point
on the open interval $(0,\frac{1}{2}r_0^2)\subset \R$ exactly once. In addition,
no other point on $\R$ is encircled, although some points lie on the image curve
$\sigma(\partial(\Omega^+))$.
But now $\gamma$ is exactly the inverse image of $\R$ under ${\sigma}$, hence it follows
that $\sigma$ is univalent on $\gamma$.
This means that $t\mapsto \sigma^{-1}(t)$, for $0<t<\frac{1}{2}r_0^2$, is a bijective analytic parametrization of $\gamma$.

Since $(\partial\Omega)^+$ is known to be analytic without singular points in the present geometry
\cite{Sakai 82}, \cite{Sakai 91}, \cite{Gustafsson-Sakai 94} the Schwarz function $S(z)$, and hence $\sigma(z)$,
extends to be analytic in a full neighborhood of $(\partial\Omega)^+$. In particular, $\sigma$ is an open
mapping in such a neighborhood.

Let $z_0\in(\partial\Omega)^+$ be a point on distance $r_0$ from the origin. Then, with the above extended $\sigma$,
$\sigma(z_0)=\frac{1}{2}r_0^2$, and any neighborhood $U$ of $z_0$ in $(\R^2)^+$ is
mapped onto a full neighborhood
$\sigma(U)$ of $\frac{1}{2}r_0^2$.
The latter neighborhood contains the interval
$(\frac{1}{2}r_0^2-\varepsilon,\frac{1}{2}r_0^2)$ on the real axis for some
$\varepsilon >0$, hence $U$ contains the final part of $\gamma$.
Since this is true for any neighborhood $U$ of $z_0$ it follows that $\overline{\gamma}\cap(\partial\Omega)^+$
consists of just the point $z_0$. In particular, there is only one point on $(\partial\Omega)^+$
on distance $r_0$ from the origin.

The last statement of the corollary follows from the fact that $L_{\boldsymbol{\xi}}(\ast \omega)=0$
(see \eqref{Lie}) in two dimensions, or simply from the analyticity of
$\frac{1}{2}zS(z)=\rho + \I \ast\omega$: $\gamma$ is a level line of $\ast\omega$
and $\boldsymbol{\xi}=\nabla \rho$ is perpendicular to level lines of $\rho$, which are orthogonal
to those of $\ast\omega$.
\end{proof}

Now we return to $n\geq 2$ dimensions. We still expect that there is only one point on $(\partial\Omega)^+$
of shortest distance to the origin, that $\gamma$ reaches $(\partial\Omega)^+$ only at that point, and that
there are no other stationary points of $\frac{1}{2}r^2$. The following proposition gives some partial results
in this direction.

\begin{proposition}\label{prop:gamma}
The closure of $\gamma$ in $\R^n$ intersects $(\partial\Omega)^+$
only at points of $(\partial\Omega)^+$
where $\boldsymbol{\xi}=0$, i.e., where $\frac{1}{2}r^2$ is stationary.
If there is a strict local maximum of $\frac{1}{2}r^2$ at some point $x\in(\partial\Omega)^+$, then
$x$ is in the closure of $\gamma$.
\end{proposition}

\begin{proof}
It is easy to see that, at points in $\Omega^+$ close to
$(\partial\Omega)^+$, $\nabla u$ is essentially directed in the
inward normal direction. At any point $x\in (\partial\Omega)^+$ where
$\boldsymbol{\xi}\ne 0$, the vector $\mathbf{x}$ from the origin to $x$ has a nonzero angle
to the normal vector. It follows from (\ref{gamma}) that $\nabla u$ cannot be parallel
to $\mathbf{x}$ close to such points, hence that $\gamma$ does not
reach points of $(\partial\Omega)^+$ where $\boldsymbol{\xi}\ne 0$.

Next, let $x\in(\partial\Omega)^+$ be a point at which $\frac{1}{2}r^2$
attains a strict local maximum. Then for some neighborhood $N$ of $x$
$$
(\overline{\Omega^+}\setminus \{x\})\cap N\subset B(0,r)\cap N
$$
($r=|x|$). For a slightly smaller $r$, say $r'<r$,
$S=(\Omega^+\cap \partial B(0,r'))\cap N$ is a piece of a sphere
cut off by $(\partial\Omega)^+$. In this piece $u>0$, while on the boundary
$u=0$. It follows that there is a point on $S$ where $u|_S$ attains a local maximum.
But at such a point the gradient $\nabla u$ is perpendicular to $S$, which means that the
point belongs to $\gamma$. Since
$r'$ was arbitrarily close to $r$, this means that we can produce points on $\gamma$
arbitrarily close to $x$.
\end{proof}


\subsection{The Hessian}

Let
$$
H=(\frac{\partial^2u}{\partial x_i \partial x_j})
$$
be the Hessian of $u$. It is a symmetric matrix with
${\rm tr\,}H =\Delta u =1$ in $\Omega^+$. The rank of
$H$ is at least one (since the sum of eigenvalues is
different from zero). On $(\partial\Omega)^+$ the rank
is exactly one since $\nabla u=0$ on $(\partial\Omega)^+$
and hence all derivatives of $\nabla u$ in tangential directions
of $(\partial\Omega)^+$ are zero.

From
$r\frac{\partial}{\partial r}\nabla u =(\mathbf{x}\cdot \nabla)\nabla u= H\mathbf{x}$,
where the last product is matrix multiplication, it follows that the definition of $\boldsymbol{\xi}$
can be written in terms of $H$ as
$$
\boldsymbol{\xi}=\mathbf{x} -(n-1)\nabla u -  H\mathbf{x}.
$$
Thus, on $(\partial\Omega)^+$,
$$
H\mathbf{x}=\mathbf{x}-\boldsymbol{\xi},
$$
and in particular
$$
H\mathbf{x}=\mathbf{x}
$$
at stationary points (for $\frac{1}{2}r^2$) on $(\partial\Omega)^+$.

More generally we can write, for any $\boldsymbol{\eta}\in\R^n$,
$$
H\boldsymbol{\eta} =\nabla(\boldsymbol{\eta}\cdot \nabla u).
$$
On $(\partial\Omega)^+$, $\boldsymbol{\eta}\cdot \nabla u =0$, hence $\nabla(\boldsymbol{\eta}\cdot \nabla u)$
is perpendicular to $(\partial\Omega)^+$. Therefore, on $(\partial\Omega)^+$,
$$
H \boldsymbol{\eta}=\lambda \mathbf{{n}}
$$
for some $\lambda=\lambda(\boldsymbol{\eta})\in\R$. It follows that
$\mathbf{{n}}$ defines the only eigendirection for $H$ corresponding to a
nonzero eigenvalue, and since the sum of eigenvalues equals one, that
$$
H\mathbf{{n}}=\mathbf{{n}}.
$$
Moreover, it follows that $H$ annihilates every vector tangent to
$(\partial\Omega)^+$. Thus, on $(\partial\Omega)^+$, $H$ simply is the
orthogonal projection onto the normal of $(\partial\Omega)^+$:
$$
H\boldsymbol{\eta} =(\boldsymbol{\eta}\cdot \mathbf{{n}})\mathbf{{n}}.
$$


\begin{proposition}
With $dx=dx_1\dots dx_n=dm$,
\begin{align*}
\int_{\Omega^+}\frac{\partial^2 u}{\partial x_i \partial x_j} dx
&=0 \quad \big((i,j)\ne (n,n)\big),\\
\int_{\Omega^+}\frac{\partial^2 u}{\partial x_n^2 } dx
&=\frac{1}{2}{m(\Omega)}=\frac{1}{2}\int d\mu.
\end{align*}

\end{proposition}

\begin{proof}
For any entry $(i,j)\ne (n,n)$, say $i=1$, $j$ arbitrary, we have
\begin{align*}
\int_{\Omega^+}\frac{\partial^2 u}{\partial x_1 \partial x_j} dx_1\dots dx_n
&=\int_{\Omega^+}d(\frac{\partial u}{ \partial x_j} dx_2\dots dx_n) \\
&=\int_{\partial(\Omega^+)}\frac{\partial u}{ \partial x_j} dx_2\dots dx_n=0
\end{align*}
(note that $dx_n =0$ on $\R^{n-1}$), while for $i=j=n$, using the above
$$
\int_{\Omega^+}\frac{\partial^2 u}{\partial x_n^2 } dx =\int_{\Omega^+}\Delta u dx=\frac{1}{2}{m(\Omega)}=\frac{1}{2}\int d\mu.
$$

\end{proof}


\subsection{Convexity questions}

We finish by some simple observations related to convexity.

\begin{proposition}
The restriction of $U^\mu$ to any hemisphere
$(\partial B (a,R))^+$ with center $a$ on $\R^{n-1}$
is convex as a function of $x'=(x_1,\dots, x_{n-1})$.
\end{proposition}

\begin{proof}

This follows from a straight-forward computation of the second derivatives
of $x'\mapsto U^\mu (x', \sqrt{R^2-|x'-a|^2})$. Assuming for simplicity that
$a=0$ and computing, e.g., the derivatives in the $x_1$-direction in the case $n\geq 3$
we have, with $x_n^2 =R^2-|x'|^2$,
\begin{align*}
U^\mu (x', \sqrt{R^2-|x'|^2})&=U^\mu (x',x_n)
=\frac{1}{(n-2)|S^{n-1}|} \int \frac{d\mu(y')}{(|x'-y'|^2+x_n^2)^{\frac{n-2}{2}}} \\
&=\frac{1}{(n-2)|S^{n-1}|} \int \frac{d\mu(y')}{(|y'|^2-2x'\cdot y'+R^2)^{\frac{n-2}{2}}}, \\
\frac{\partial U^\mu}{\partial x_1}
&=\frac{1}{(n-2)|S^{n-1}|} \int \frac{(n-2)y_1\, d\mu(y')}{(|y'|^2-2x'\cdot y'+R^2)^{\frac{n}{2}}}, \\
\frac{\partial^2 U^\mu}{\partial x_1^2}
&=\frac{1}{(n-2)|S^{n-1}|} \int \frac{n(n-2)y_1^2\, d\mu(y')}{(|y'|^2-2x'\cdot y'+R^2)^{\frac{n+2}{2}}}\geq 0,
\end{align*}
where $|S^{n-1}|$ denotes the spherical area of $S^{n-1}$.

\end{proof}

\begin{proposition}
Assume that $\supp \mu\subset {B(a,R)}\cap \R^{n-1}$ for a ball $B(a,R)$
with center $a\in\R^{n-1}$. Then the density of classical balayage of $\mu$
onto the sphere $\partial B(a,R)$ is convex as a function of $x'$.

\end{proposition}

\begin{proof}

We may assume that $a=0$. Set $B_R=B(0,R)$.
Classical balayage of a point mass at any point $y\in B_R$
onto $\partial B_R$
gives the mass distribution on $\partial B_R$ with density
equal to the Poisson kernel (or normal derivative of Green
function), i.e., the function of $x\in \partial B_R$
$$
P_{B_R}(x, y)
=\frac{1}{|S^{n-1}|R}\frac{R^2 - |y|^2}{|x-y|^n}
=\frac{1}{|S^{n-1}|R}\frac{R^2 - |y|^2}{(|y|^2- 2x\cdot y+R^2)^{\frac{n}{2}}}.
$$
It follows that the density $\beta_0$ of classical balayage of $\mu$ is obtained by integrating the
above with respect to $y'\in \R^{n-1}$:
\begin{align*}
\beta_0(x,B_R)
&=\frac{1}{|S^{n-1}|R}\int \frac{R^2 - |y'|^2}{(|y'|^2- 2x\cdot y'+R^2)^{\frac{n}{2}}}d\mu(y') \\
&=\frac{1}{|S^{n-1}|R}\int \frac{R^2 - |y'|^2}{(|y'|^2- 2x'\cdot y'+R^2)^{\frac{n}{2}}}d\mu(y').
\end{align*}

Note that the last expression is a function only of $x'$.
One sees immediately that as such a function the second order derivative in any direction
is nonnegative (compare previous proof).
In fact, any derivative of even order is nonnegative (a similar remark applies to the previous proposition).

\end{proof}


\end{document}